\documentclass[11pt]{amsart}
\usepackage[T1]{fontenc}
\usepackage{amsmath,amssymb,amsthm,mathtools}
\usepackage[margin=1.05in]{geometry}
\usepackage{microtype,booktabs,array}
\usepackage[hidelinks]{hyperref}
\makeatletter
\@ifundefined{subjclassname@2020}{%
  \@namedef{subjclassname@2020}{\textup{2020} Mathematics Subject Classification}%
}{}
\makeatother

\newtheorem{theorem}{Theorem}[section]
\newtheorem{proposition}[theorem]{Proposition}
\newtheorem{lemma}[theorem]{Lemma}
\newtheorem{corollary}[theorem]{Corollary}
\theoremstyle{definition}
\newtheorem{definition}[theorem]{Definition}
\newtheorem{conjecture}[theorem]{Conjecture}
\theoremstyle{remark}
\newtheorem{remark}[theorem]{Remark}
\newtheorem{example}[theorem]{Example}
\newcommand{\CC}{\mathbb C}

\newcommand{\OO}{\mathcal O}
\newcommand{\mm}{\mathfrak m}
\newcommand{\nn}{\mathfrak n}
\newcommand{\ee}{\mathbf e}
\newcommand{\lgen}[1]{\lambda^{#1}_{\mathrm{gen}}}
\newcommand{\eHS}{e_{\mathrm{HS}}}
\newcommand{\jmult}{\mathrm j}
\DeclareMathOperator{\mult}{mult}
\DeclareMathOperator{\Frac}{Frac}
\DeclareMathOperator{\Min}{Min}

\DeclareMathOperator{\Spec}{Spec}
\DeclareMathOperator{\Tr}{Tr}

\DeclareMathOperator{\htop}{ht}

\numberwithin{equation}{section}
\hypersetup{
  pdftitle={Finite holomorphic pullbacks, Segre numbers, and generic Le numbers},
  pdfauthor={Z. Jelonek; V. H. Jorge Perez; R. N. Araújo dos Santos},
  pdfsubject={Finite holomorphic maps and invariants of hypersurface singularities}
}

\title[Finite pullbacks and Segre numbers]
{Finite holomorphic pullbacks, Segre numbers,\protect\\
and generic L\^e numbers}
\author{Z. Jelonek}
\address{Instytut Matematyczny, Polska Akademia Nauk,
\'Sniadeckich 8, 00-656 Warszawa, Poland}
\email{najelone@cyf-kr.edu.pl}
\author{V. H. Jorge Perez}
\address{Universidade de S{\~a}o Paulo, ICMC,
13560-970, S{\~a}o Carlos-SP, Brazil}
\email{vhjperez@icmc.usp.br}
\author{R. N. Araújo dos Santos}
\address{Universidade de S{\~a}o Paulo, ICMC,
13560-970, S{\~a}o Carlos-SP, Brazil}
\email{rnonato@icmc.usp.br}
\thanks{The first author is partially supported by Narodowe Centrum
Nauki, grant no.~2024/55/B/ST1/01412. The second author CNPq-Brazil, grant 304851/2025-6}
\date{}

\begin{document}

\begin{abstract}
We study numerical comparisons for hypersurface singularities under
finite holomorphic maps between smooth germs. We retain Jelonek's
topological proof of the Milnor-number inequality in the isolated case
and give an algebraic proof using a cyclic factorization, Brieskorn
modules, and trace. For a reduced divisor-theoretic pullback, we prove
lexicographic monotonicity of Segre numbers. Componentwise
monotonicity follows when the two singular loci have the same
dimension and the target Jacobian ideal is equimultiple. For
one-dimensional singular loci, we also obtain componentwise
comparison under an equality condition on the first nonzero Segre
number or a nonramification condition. The arguments use
finite-flat pullback of Segre cycles and Hilbert--Samuel
multiplicity. Through the identification with generic L\^e numbers,
the same results hold for these invariants. We distinguish these
statements from comparisons in arbitrary prepolar coordinates and
formulate the remaining componentwise and ramified cases as
conjectures.
\end{abstract}

\maketitle

\section{Introduction}

Let \(F:(\CC^n,0)\to(\CC^n,0)\) be a finite holomorphic map germ.
Giraldo and Roeder proved that the inverse image of a singular
hypersurface is singular \cite{GR}. For irreducible hypersurfaces
with isolated singularities, Jelonek obtained the quantitative
comparison
\[
\mu(h)\geq\mu(g)
\]
when \(V(h)=F^{-1}(V(g))_{\mathrm{red}}\)
\cite[Theorem~2.2]{Jelonek}. Segre numbers and generic L\^e numbers
provide natural numerical candidates for extending this comparison
to nonisolated singularities.

A necessary distinction is the one between a reduced inverse image
and a reduced pullback divisor. If \(g,h\) are reduced equations of
irreducible hypersurfaces and
\(V(h)=F^{-1}(V(g))_{\mathrm{red}}\), unique factorization gives
\begin{equation}\label{eq:standing-power}
g\circ F=u h^r,\qquad u\in\OO_{\CC^n,0}^{\times},\quad r\geq1.
\end{equation}
The pullback divisor is reduced precisely when \(r=1\). After
absorbing a root of \(u\) into \(h\), the chain rule gives
\[
J(h)\subseteq F^*J(g)\quad\text{if }r=1,
\qquad
h^{r-1}J(h)\subseteq F^*J(g)\quad\text{in general}.
\]
These inclusions have substantially different consequences.

The paper has three principal comparison results. First, in the
isolated case we present both Jelonek's topological argument and
an algebraic argument valid for every \(r\); when \(r=1\), the
stronger bound \(\mu(h)\geq (\deg F)\mu(g)\) follows directly
from Hilbert--Samuel multiplicity. Second, for \(r=1\), Segre
vectors are nondecreasing in lexicographic order, with no restriction
on the dimensions of the singular loci
(Theorem~\ref{thm:reduced-lex}). Third, if these dimensions agree
and \(J(g)\) is equimultiple, the comparison is componentwise
(Theorem~\ref{thm:equimultiple}). Additional componentwise cases
for one-dimensional singular loci appear in
Corollary~\ref{cor:curve-cases}.

The algebraic input is a finite-flat comparison for an arbitrary
ideal \(J\) in a smooth target:
\[
e_j(J)\leq e_j(F^*J)\leq (\deg F)e_j(J),
\qquad e_n(F^*J)=(\deg F)e_n(J).
\]
We prove this using compatible Segre cycles and a graph argument
for Hilbert--Samuel multiplicity. The final passage from
\(F^*J(g)\) to \(J(h)\) uses the lexicographic theorem of
Cid-Ruiz, Polini, and Ulrich \cite[Corollary~6.5]{CPU}; an ideal
inclusion alone does not give componentwise monotonicity.

We include precise definitions and indexing conventions for Segre
cycles, multiplicity sequences, and L\^e cycles. In particular,
the correspondence with Segre numbers concerns \emph{generic}
L\^e numbers, not arbitrary coordinate-dependent L\^e numbers.
The full componentwise statement, including higher ramification,
is left as a conjecture in Section~\ref{sec:conjectures}.

\section{Cycles and numerical invariants}\label{sec:definitions}

\subsection{Notation and multiplicity sequences}

Throughout the paper \(n\geq2\) is the ambient dimension and
\(d\) is the degree of \(F\). We use
\[
A=\CC\{y_1,\ldots,y_n\},\qquad
B=\CC\{x_1,\ldots,x_n\},
\]
with maximal ideals \(\mm\) and \(\nn\), for the target and source
rings. The \(A\)-algebra structure on \(B\) is induced by \(F\).
For a holomorphic germ \(f\), set
\[
J(f)=\left(\frac{\partial f}{\partial x_1},\ldots,
\frac{\partial f}{\partial x_n}\right),\qquad
\Sigma(f)=V(J(f)).
\]
After shrinking a representative, the critical locus of a reduced
germ through the origin is contained in its zero fiber. We always
compute the Jacobian invariants in the \emph{ambient} smooth ring,
not in \(\OO_{V(f),0}\).

For a finite module \(M\) and an ideal \(L\) with
\(M/LM\) of finite length, let \(\eHS(L;M)\) denote
Hilbert--Samuel multiplicity. The notation \(\eHS(R)\) means
multiplicity with respect to the maximal ideal of the local ring
\(R\), and \(\lambda_R\) denotes length.

Let \((R,\mathfrak a)\) be an \(n\)-dimensional local analytic
ring and \(0\ne Q\subsetneq R\). Put
\[
G_Q(R)=\bigoplus_{v\geq0}Q^v/Q^{v+1},\qquad
T=G_{\mathfrak a}(G_Q(R)).
\]
Thus
\[
T_{u,v}=
\frac{\mathfrak a^uQ^v+Q^{v+1}}
{\mathfrak a^{u+1}Q^v+Q^{v+1}}.
\]
The double cumulative Hilbert function is
\begin{equation}\label{eq:double-hilbert}
H_Q^{(1,1)}(u,v)=
\sum_{a=0}^u\sum_{b=0}^v\lambda_R(T_{a,b})
=\sum_{b=0}^v\lambda_R
\left(\frac{Q^b}{\mathfrak a^{u+1}Q^b+Q^{b+1}}\right).
\end{equation}
For sufficiently large \(u,v\), it agrees with a polynomial whose
homogeneous part of total degree \(n\) is
\begin{equation}\label{eq:multiplicity-sequence}
\sum_{i=0}^n\frac{c_i(Q;R)}{i!(n-i)!}\,u^i v^{n-i}.
\end{equation}
This defines the multiplicity sequence in the
\emph{dimension indexing} of Achilles--Manaresi \cite{AM97,AM}.
All its coefficients are nonnegative integers.

The analytic spread and equimultiplicity condition are
\[
\ell(Q)=\dim\left(\bigoplus_{v\geq0}Q^v/\mathfrak aQ^v\right),
\qquad Q\text{ equimultiple}\iff\ell(Q)=\htop Q.
\]
We use \(\ell(Q)\), rather than \(s(Q)\), to avoid confusion
with the dimension of a singular locus.

\begin{theorem}[Achilles--Manaresi]\label{thm:AM}
Let \(R\) be regular local analytic of dimension \(n\), let
\(0\ne Q\subsetneq R\), and set \(q=\dim R/Q\). Then
\begin{align}
c_i(Q;R)&=0\quad\text{if }i<n-\ell(Q)\text{ or }i>q,
\label{eq:AM-vanishing}\\
c_q(Q;R)&=
\sum_{\substack{\mathfrak p\in\Min(Q)\\\dim R/\mathfrak p=q}}
\eHS(QR_{\mathfrak p};R_{\mathfrak p})\,
\eHS(R/\mathfrak p).\label{eq:AM-assoc}
\end{align}
If \(Q\) is equimultiple, only \(c_q(Q;R)\) can be
nonzero. In particular, if \(Q\) is primary to the maximal
ideal, then
\[
(c_0(Q;R),\ldots,c_n(Q;R))=(\eHS(Q;R),0,\ldots,0).
\]
\end{theorem}

\begin{proof}
These are parts (1), (3), and (5) of \cite[Theorem~3]{AM}.
Regularity supplies the dimension formula required in the
associativity statement. Thus its highest-dimensional associated
primes are exactly the primes displayed in \eqref{eq:AM-assoc}.
\end{proof}

\subsection{Polar cycles and Segre cycles}

Let \((Z,0)\subset(\CC^N,0)\) be reduced and pure of
dimension \(m\), and let \(Q\subset\OO_{Z,0}\) define a
nowhere dense subspace. A cycle is a locally finite sum
\(C=\sum a_\alpha[Z_\alpha]\) of irreducible analytic subspaces
with integer coefficients. For an effective pure-dimensional
cycle through the origin, put
\[
\mult_0 C=\sum a_\alpha\,\eHS(\OO_{Z_\alpha,0}).
\]
For a Cartier divisor meeting a cycle properly, intersections
are taken with their scheme-theoretic multiplicities. On an
irreducible component of an intersection, the coefficient is the
length of the corresponding localized quotient ring.

Choose generators \(q_1,\ldots,q_a\) of \(Q\) and sufficiently
general constant linear combinations \(a_1,\ldots,a_m\) of
these generators. Starting with \(P_0=[Z]\), decompose successively
\begin{equation}\label{eq:segre-cycle-recursion}
P_{j-1}\cdot V(a_j)=P_j+\Lambda_j(Q,Z),\qquad 1\leq j\leq m.
\end{equation}
Here \(\Lambda_j(Q,Z)\) is the sum of the components supported
in \(V(Q)\), and \(P_j\) is the sum of the other components.
Both cycles retain their intersection multiplicities. The
underlying analytic set of \(P_j\) is
\[
\overline{\bigl(|P_{j-1}|\cap V(a_j)\bigr)\setminus V(Q)}.
\]
The cycles \(P_j\) and \(\Lambda_j\) are respectively polar
cycles and Segre cycles of dimension \(m-j\). Zero cycles in
the sense of the zero element of a cycle group are permitted.

\begin{definition}[Segre numbers]\label{def:segre}
The Segre numbers of \(Q\) at the origin are
\[
e_j(Q,Z)=\mult_0\Lambda_j(Q,Z),\qquad 1\leq j\leq m,
\]
for a sufficiently general choice in
\eqref{eq:segre-cycle-recursion}.
\end{definition}

Their values are independent of that general choice and of the
chosen generators; see \cite[Section~2]{GG}. We may therefore
use nonminimal generating sets. When the ambient ring \(R\)
is regular of dimension \(n\), Achilles and Rams give
\begin{equation}\label{eq:segre-c}
e_j(Q;R)=c_{n-j}(Q;R),\qquad 1\leq j\leq n;
\end{equation}
see \cite[Theorem~2]{AR} and \cite[Theorem~11]{AM}. In particular,
\(e_n(Q;R)=c_0(Q;R)=\jmult(Q;R)\), the \(j\)-multiplicity.

\subsection{L\^e cycles, admissible coordinates, and generic values}

Let \(f:(\CC^n,0)\to(\CC,0)\) be nonconstant, and let
\(\mathbf w=(w_1,\ldots,w_n)\) be local coordinates. The following
description is the cycle-theoretic construction of L\^e cycles
\cite[Lecture~3]{MasseyNotes}. Start with
\(\Gamma^n_{f,\mathbf w}=[\CC^n]\). For \(i=n-1,\ldots,0\),
decompose
\begin{equation}\label{eq:le-cycle-recursion}
\Gamma^{i+1}_{f,\mathbf w}\cdot
V\left(\frac{\partial f}{\partial w_{i+1}}\right)
=\Gamma^i_{f,\mathbf w}+\Lambda^i_{f,\mathbf w},
\end{equation}
where \(\Lambda^i_{f,\mathbf w}\) consists of the components
contained in \(\Sigma(f)\), while \(\Gamma^i_{f,\mathbf w}\)
consists of the remaining components. Each intersection must be
proper; the resulting cycles have dimension \(i\).

If the additional intersections at the origin are isolated, define
\begin{equation}\label{eq:le-number}
\lambda^i_{f,\mathbf w}(0)=
\bigl(\Lambda^i_{f,\mathbf w}\cdot V(w_1,\ldots,w_i)\bigr)_0,
\qquad 0\leq i\leq n-1.
\end{equation}
For \(i=0\), no coordinate hyperplane is imposed: the right side
is the coefficient of \([0]\) in \(\Lambda^0_{f,\mathbf w}\).
Coordinates satisfying these cycle and intersection conditions
will be called admissible here. In particular,
\(\lambda^i_{f,\mathbf w}(0)=0\) for \(i>\dim\Sigma(f)\).

Prepolar coordinates give a standard stratified-geometric way to
ensure the required properness. One fixes a good stratification
adapted to \(f\), and requires each successive coordinate
hyperplane to be transverse, away from the origin, to the induced
strata on the preceding coordinate slice; see \cite{Massey95}.
This is a transversality condition, not a statement that all the
numerical intersections equal their generic values. In particular,
``prepolar'' and ``generic'' must not be used interchangeably.

\begin{definition}[Generic L\^e numbers]\label{def:generic-le}
Fix linear coordinates \(\mathbf x\) on \(\CC^n\). There is a
nonempty Zariski-open subset \(\mathcal U_f\subset\mathrm{GL}_n(\CC)\)
on which the coordinates \(M\mathbf x\) are sufficiently general
and all the L\^e numbers at the origin have constant values.
Define
\begin{equation}\label{eq:generic-le}
\lgen{i}(f)=\lambda^i_{f,M\mathbf x}(0),\qquad
M\in\mathcal U_f,\quad 0\leq i\leq n-1.
\end{equation}
The open set may be shrunk without changing these values.
\end{definition}

Generic choices are understood in the usual local polar-variety
sense; see \cite{Massey95,GG}, and the distinction between
specified coordinates and generic values in \cite[Introduction]{Massey24}.
Changing the initial linear coordinates only translates the open
set in \(\mathrm{GL}_n(\CC)\), so the definition is independent
of that choice.

\begin{remark}[Generic L\^e numbers and Segre numbers]
\label{rem:generic-identification}
For the ambient Jacobian ideal,
\begin{equation}\label{eq:le-segre-c}
\lgen{i}(f)=e_{n-i}(J(f);\CC^n)=c_i(J(f)),
\qquad 0\leq i\leq n-1.
\end{equation}
The first equality identifies the generic L\^e construction with
the Segre construction for the gradient ideal; see
\cite[p.~346]{AM}, \cite{GG,Massey95}. The second equality is
\eqref{eq:segre-c}. This does not identify the L\^e numbers in
every admissible or prepolar coordinate system with Segre numbers.
\end{remark}

\begin{remark}[Isolated singularities]\label{rem:isolated-identification}
If \(\Sigma(f)=\{0\}\), then \(J(f)\) is a parameter ideal in
the regular local ring \(\OO_{\CC^n,0}\). Hence
\[
\mu(f)=\dim_\CC\OO_{\CC^n,0}/J(f)
=\eHS(J(f);\OO_{\CC^n,0}),
\]
and
\[
\lgen{0}(f)=e_n(J(f))=\mu(f),\qquad
\lgen{i}(f)=0\quad(i\geq1).
\]
In fact \(\lambda^0_{f,\mathbf w}(0)=\mu(f)\) in every
coordinate system for which the construction is defined
\cite[Example~3.11]{MasseyNotes}.
\end{remark}

\subsection{Units and lexicographic order}

\begin{lemma}[Independence of a unit factor]\label{lem:units}
If \(v\) is a holomorphic unit, then
\(\overline{J(vf)}=\overline{J(f)}\). Consequently the Segre
numbers and generic L\^e numbers of \(vf\) and \(f\) agree.
\end{lemma}

\begin{proof}
The analytic curve criterion gives
\(f\in\overline{\mm J(f)}\): along an analytic arc through
the origin, differentiating \(f\) along the arc proves the
required order inequality. By the Leibniz rule,
\(J(vf)\subseteq\overline{J(f)}\). Applying the same argument
to \(v^{-1}\) proves the reverse inclusion of integral closures.
Segre numbers are invariant under integral closure
\cite{GG,AR}; use \eqref{eq:le-segre-c} for generic L\^e numbers.
\end{proof}

For \(\mathbf a,\mathbf b\in\mathbb Z^n\), write
\(\mathbf a\geq_{\mathrm{lex}}\mathbf b\) if the vectors are
equal or their first unequal entries satisfy \(a_k>b_k\).
Set
\[
\ee(Q;R)=(e_1(Q;R),\ldots,e_n(Q;R)).
\]
Thus Segre vectors use increasing codimension, whereas the
corresponding L\^e vector is
\((\lgen{n-1}(f),\ldots,\lgen{0}(f))\).

\begin{theorem}[Cid-Ruiz--Polini--Ulrich]\label{thm:ideal-lex}
For proper nonzero ideals \(K\subseteq L\) of a regular local
analytic ring \(R\),
\[
\ee(K;R)\geq_{\mathrm{lex}}\ee(L;R).
\]
\end{theorem}

\begin{proof}
Apply \cite[Corollary~6.5]{CPU}. A regular local analytic ring is
equidimensional and universally catenary. That paper uses
codimension indexing: its \(c_j\), for \(j\geq1\), is our
\(e_j\); its codimension-zero term is zero for the nonzero
ideals under consideration.
\end{proof}

\section{Finite-flat pullbacks and multiplicity}\label{sec:flat}

\begin{lemma}[Freeness and primary multiplicity]\label{lem:free}
Let \(F:(\CC^n,0)\to(\CC^n,0)\) be finite of degree \(d\).
Then \(B\) is free of rank \(d\) over \(A\). For every
\(\mm\)-primary ideal \(Q\subset A\),
\begin{equation}\label{eq:primary-degree}
\eHS(QB;B)=d\,\eHS(Q;A).
\end{equation}
More generally, the freeness assertion and
\eqref{eq:primary-degree} remain valid for a Cohen--Macaulay
source of dimension \(n\) and a smooth target, with \(d\)
interpreted as the generic module rank.
\end{lemma}

\begin{proof}
Since \(F\) is finite, \(B\) is finite over \(A\) and
\(\sqrt{\mm B}=\nn\). The images of a regular system of
parameters of \(A\) form a system of parameters of \(B\).
Cohen--Macaulayness makes it a \(B\)-regular sequence. Hence
\(\operatorname{depth}_A B=n\); the Auslander--Buchsbaum
formula over the regular ring \(A\) gives
\(\operatorname{pd}_A B=0\). A finite projective module over
the local ring \(A\) is free, and its rank is \(d\).

For every \(k\geq1\),
\[
B/Q^kB\cong(A/Q^k)\otimes_A B\cong(A/Q^k)^{\oplus d}
\]
as \(A\)-modules. Both residue fields are \(\CC\), so the
\(A\)-length and \(B\)-length of a finite-length \(B\)-module
are equal. Comparing the leading Hilbert--Samuel coefficients
proves \eqref{eq:primary-degree}.
\end{proof}

\begin{lemma}[A graph multiplicity bound]\label{lem:graph}
For every proper ideal \(Q\subset A\),
\begin{equation}\label{eq:graph-bound}
\eHS(A/Q)\leq\eHS(B/QB)\leq d\,\eHS(A/Q).
\end{equation}
\end{lemma}

\begin{proof}
Put \(t=\dim A/Q=\dim B/QB\), where equality follows from
finiteness and faithful flatness. We first record a parameter
inequality. If \((R,\mathfrak a)\) has dimension \(D>0\),
\(z\in\mathfrak a\), and \(\dim R/(z)=D-1\), then
\begin{equation}\label{eq:parameter-bound}
\eHS(R/(z))\geq\eHS(R).
\end{equation}
Indeed, the right-exact sequence
\[
R/\mathfrak a^k\xrightarrow{\cdot z}R/\mathfrak a^{k+1}
\longrightarrow R/(\mathfrak a^{k+1}+zR)\longrightarrow0
\]
gives
\(H_{R/(z)}(k)\geq H_R(k)-H_R(k-1)\), where
\(H_R(k)=\lambda_R(R/\mathfrak a^{k+1})\).
Comparison of leading coefficients proves
\eqref{eq:parameter-bound}, including the case \(D=1\).

Now set
\[
T=\CC\{y_1,\ldots,y_n,x_1,\ldots,x_n\}/
Q\CC\{y_1,\ldots,y_n,x_1,\ldots,x_n\}.
\]
Then \(\dim T=t+n\), \(\eHS(T)=\eHS(A/Q)\), and
\[
T/(y_1-F_1(x),\ldots,y_n-F_n(x))\cong B/QB.
\]
The \(n\) graph equations cut dimension by \(n\); every
successive quotient therefore cuts dimension by one. Applying
\eqref{eq:parameter-bound} successively proves the lower bound
in \eqref{eq:graph-bound}.

For the upper bound, use \(\mm B\subseteq\nn\) and the
fact that \(B/QB\) is a free \(A/Q\)-module of rank \(d\).
Length comparison gives
\[
\eHS(\nn;B/QB)\leq\eHS(\mm B;B/QB)
=d\,\eHS(\mm;A/Q),
\]
as required.
\end{proof}

\begin{proposition}[Pullback of Segre cycles]\label{prop:flat-cycles}
Let \(0\ne J\subsetneq A\), and put \(I=JB\). A common
sufficiently general sequence of linear combinations of generators
of \(J\), and their pullbacks, computes the Segre numbers of
\(J\) and \(I\). For these choices the cycles satisfy
\begin{equation}\label{eq:flat-cycles}
\Lambda_j(I,\CC^n)=F^*\Lambda_j(J,\CC^n),
\qquad 1\leq j\leq n,
\end{equation}
where \(F^*\) denotes flat pullback of cycles, including
multiplicities.
\end{proposition}

\begin{proof}
Use a fixed set of generators of \(J\) and its image in \(B\).
The good coefficient matrices for each Segre construction contain
a nonempty Zariski-open subset of the same matrix space. Their
intersection is nonempty. Choose a matrix in this intersection.

Flat pullback commutes with intersection by a Cartier divisor.
Moreover, a component of the inverse image of an irreducible
cycle component maps onto that component: the map is finite flat
of relative dimension zero. Such an inverse-image component is
contained in \(V(I)\) exactly when its image is contained in
\(V(J)\). Consequently flat pullback commutes with both parts
of the decomposition \eqref{eq:segre-cycle-recursion}.
Starting from \(F^*[\CC^n]=[\CC^n]\), induction gives
\(P_j(I)=F^*P_j(J)\) and \eqref{eq:flat-cycles}.
This argument concerns cycles, not merely their reduced supports;
ramification multiplicities are retained throughout.
\end{proof}

\begin{theorem}[Finite-flat Segre comparison]\label{thm:flat-segre}
Let \(0\ne J\subsetneq A\), let \(I=JB\), and let \(d=\deg F\).
Then
\begin{equation}\label{eq:flat-segre-bounds}
e_j(J;A)\leq e_j(I;B)\leq d\,e_j(J;A),
\qquad 1\leq j\leq n,
\end{equation}
and the last Segre number satisfies
\begin{equation}\label{eq:j-degree}
e_n(I;B)=d\,e_n(J;A).
\end{equation}
Equivalently,
\(c_i(J;A)\leq c_i(I;B)\leq d\,c_i(J;A)\), with
equality on the right for \(i=0\).
\end{theorem}

\begin{proof}
Choose compatible cycles as in Proposition~\ref{prop:flat-cycles}
and write
\[
\Lambda_j(J)=\sum_\alpha a_\alpha[V(\mathfrak p_\alpha)],
\qquad \dim A/\mathfrak p_\alpha=n-j.
\]
The flat pullback of \([V(\mathfrak p_\alpha)]\) is the
fundamental cycle of \(B/\mathfrak p_\alpha B\). By the
associativity formula for Hilbert--Samuel multiplicity,
\begin{align*}
e_j(J;A)&=\sum_\alpha a_\alpha\,\eHS(A/\mathfrak p_\alpha),\\
e_j(I;B)&=\sum_\alpha a_\alpha\,\eHS(B/\mathfrak p_\alpha B).
\end{align*}
Lemma~\ref{lem:graph}, applied term by term, proves
\eqref{eq:flat-segre-bounds}. For \(j=n\), the only possible
component downstairs is \([0]\), and its flat pullback has
multiplicity \(\lambda_B(B/\mm B)=d\). This gives
\eqref{eq:j-degree}.
\end{proof}

\begin{remark}[Why the degree formula does not hold at every index]
\label{rem:false-degree}
Freeness gives the expected factor \(d\) when the auxiliary
filtration upstairs is the \(\mm B\)-adic filtration.
It does not identify that filtration with the
\(\nn\)-adic filtration used in the definition of \(c_i(I;B)\).
Although \(\sqrt{\mm B}=\nn\), this does not imply that
\(\mm B\) is a reduction of \(\nn\).

For example, let \(F(u,v)=(u,v^d)\), with \(d>1\), and
\(J=(y_1)\). Then \(I=(u)\) and
\[
e_1(J;A)=e_1(I;B)=1,
\]
not \(e_1(I;B)=d\). Here \(\mm B=(u,v^d)\), whose
Hilbert--Samuel multiplicity is \(d\), whereas
\(\eHS(\nn;B)=1\). Thus these primary ideals cannot have
the same integral closure. The exact formula that survives for
all ideals is \eqref{eq:j-degree}, not an exact formula at
every index.
\end{remark}

\begin{lemma}[The first possible coefficient under inclusion]
\label{lem:first-coefficient}
Let \(R\) be regular local analytic of dimension \(n\), and
let \(K\subseteq L\) be proper nonzero ideals with
\(\dim R/K=\dim R/L=s\). Then
\begin{equation}\label{eq:first-coefficient}
c_s(K;R)\geq c_s(L;R).
\end{equation}
If \(L\) is equimultiple, then
\(c_i(K;R)\geq c_i(L;R)\) for every \(i\).
\end{lemma}

\begin{proof}
Every minimal prime of \(L\) with quotient dimension \(s\)
is a minimal prime of \(K\) with that dimension, because
\(K\subseteq L\) and their heights agree. At such a prime
\(\mathfrak p\), both localized ideals are primary and
\[
\eHS(KR_{\mathfrak p};R_{\mathfrak p})
\geq\eHS(LR_{\mathfrak p};R_{\mathfrak p}).
\]
Formula~\eqref{eq:AM-assoc}, together with the nonnegative
contributions of any additional components of \(V(K)\), proves
\eqref{eq:first-coefficient}. If \(L\) is equimultiple,
all its other coefficients vanish by Theorem~\ref{thm:AM}.
\end{proof}

\section{Isolated singularities: two proofs}\label{sec:isolated}

We first record the degree calculation used in the algebraic
argument. In this section a unit in \eqref{eq:standing-power}
is absorbed into \(h\), using Lemma~\ref{lem:units}, so that
\begin{equation}\label{eq:normalized-power}
g\circ F=h^r.
\end{equation}

\begin{lemma}[Cyclic factorization and degrees]\label{lem:cyclic}
Suppose \(g,h\) are irreducible and reduced and
\eqref{eq:normalized-power} holds. The analytic algebra
\[
C=A[z]/(z^r-g)
\]
embeds in \(B\) by \(z\mapsto h\), and \(B\) is finite over
\(C\). Moreover,
\begin{equation}\label{eq:field-degree}
[\Frac(B):\Frac(C)]=\frac d r.
\end{equation}
In particular \(r\mid d\). The finite map of reduced
hypersurfaces \(F|_{V(h)}:V(h)\to V(g)\) also has degree
\(d/r\).
\end{lemma}

\begin{proof}
The ring \(A_{(g)}\) is a discrete valuation ring, and
\(z^r-g\) is Eisenstein over it. It is therefore irreducible
over \(\Frac(A)\), and \(C\) is a domain of rank \(r\) over
\(A\). Evaluation at \(h\) is injective: after passing to
\(\Frac(A)\) it is a nonzero map from a field into
\(\Frac(B)\). A finite set of \(A\)-module generators of
\(B\) also generates it over \(C\). The tower law now gives
\[
d=[\Frac(B):\Frac(A)]
=[\Frac(B):\Frac(C)]\,[\Frac(C):\Frac(A)],
\]
which proves \eqref{eq:field-degree}.

For the last assertion, localize the finite extension at \((g)\).
There is only one height-one prime over \((g)\), namely
\((h)\), and its ramification index is \(r\). The
free module of rank \(d\) over \(A_{(g)}\), reduced modulo
\(g\), has dimension \(d\) over the residue field. Computing
this dimension with the \(h\)-adic filtration of length \(r\)
gives
\[
d=r\,[\Frac(B/(h)):\Frac(A/(g))].
\]
The residual field degree on the right is the generic degree of
the hypersurface map. No formula involving a ratio of the
multiplicities of \(V(g)\) and \(V(h)\) is used here.
\end{proof}

\begin{theorem}[Jelonek's isolated-singularity comparison]
\label{thm:isolated}
Let \(F:(\CC^n,0)\to(\CC^n,0)\) be finite. Let \(g,h\)
be reduced equations of irreducible hypersurface germs with at most isolated singularities, where
\[
V(h)=F^{-1}(V(g))_{\mathrm{red}},\qquad
\]
Then
\begin{equation}\label{eq:milnor-bound}
\mu(h)\geq\mu(g).
\end{equation}
Equivalently, \(\eHS(J(h);B)\geq\eHS(J(g);A)\).
Consequently all Segre numbers, and all generic L\^e numbers,
are componentwise nondecreasing. This holds for every
\(r\geq1\) in \eqref{eq:standing-power}.
\end{theorem}

\begin{proof}[Topological proof (Jelonek)]
We give the argument of \cite[Theorem~2.2]{Jelonek} with field
coefficients explicit. Normalize as in
\eqref{eq:normalized-power}. Choose a finite proper
representative of \(F\). Choose a small target Milnor ball
\(B_\varepsilon\) and a source Milnor ball \(B'\) such that
\(F(B')\subset B_\varepsilon\). Next choose a target Milnor
ball \(B_\eta\), with \(\eta<\varepsilon\), such that
\(F^{-1}(B_\eta)\subset B'\).

For a sufficiently small nonzero \(t\), put \(c=t^r\) and
use the interiors of the corresponding Milnor fibers:
\[
M_g^- =g^{-1}(c)\cap B_\eta^{\circ},\qquad
M_g^+ =g^{-1}(c)\cap B_\varepsilon^{\circ},\qquad
M_h=h^{-1}(t)\cap (B')^{\circ}.
\]
The inclusion \(\iota:M_g^-\hookrightarrow M_g^+\) is a
homotopy equivalence. The space
\[
N=F^{-1}(M_g^-)\cap h^{-1}(t)
\]
is an open submanifold of \(M_h\), and
\(q=F|_N:N\to M_g^-\) is a nonempty finite proper map onto
the connected target. In particular it has positive degree.

A proper surjective holomorphic map between complex manifolds
of the same dimension induces an injection on cohomology with
\(\CC\)-coefficients \cite[Theorem~3.1]{Wells}. Equivalently,
one may use the positive-degree transfer with rational
coefficients. Thus \(q^*\) is injective. Let
\(k=F|_{M_h}:M_h\to M_g^+\) and let \(\jmath:N\hookrightarrow M_h\)
be the inclusion. The identity
\[
\jmath^*k^*=q^*\iota^*
\]
shows that \(k^*:H^{n-1}(M_g^+;\CC)\to H^{n-1}(M_h;\CC)\)
is injective, since \(\iota^*\) is an isomorphism.

By Milnor's bouquet theorem \cite{Milnor}, these vector spaces
have dimensions \(\mu(g)\) and \(\mu(h)\). This proves
\eqref{eq:milnor-bound}. Notice that \(N\) itself was not
assumed to be a Milnor fiber, and no surjectivity over
\(\mathbb Z\) is required.
\end{proof}

\begin{proof}[Algebraic proof via Brieskorn modules]
Normalize the unit and use the cyclic algebra \(C\) of
Lemma~\ref{lem:cyclic}. Set
\[
A^\circ=A[1/g],\qquad C^\circ=C[1/z],\qquad B^\circ=B[1/h],
\qquad \delta=d/r.
\]
The extension \(A^\circ\subset C^\circ\) is finite \'etale,
since \(rz^{r-1}\) is invertible. Both \(C^\circ\) and
\(B^\circ\) are regular, and the finite map
\(\Phi:\Spec B^\circ\to\Spec C^\circ\) has degree \(\delta\).

For an isolated hypersurface singularity \(f\), its Brieskorn
module is
\[
\mathcal B_f=
\frac{\Omega^n_{\CC^n,0}}
{df\wedge d\Omega^{n-2}_{\CC^n,0}}.
\]
Here and below \(\Omega^p\) means holomorphic differential
forms, with localization when indicated. Multiplication by
\(f\) defines the action of the parameter ring \(\CC\{t\}\).
The Brieskorn--Sebastiani theorem says that \(\mathcal B_f\)
is free over \(\CC\{t\}\) of rank \(\mu(f)\);
see \cite[Section~1]{Sabbah}.

Write \(K_t=\Frac\CC\{t\}\), \(K_z=\Frac\CC\{z\}\), with
\(t=z^r\), and define
\[
Q_C=\frac{\Omega^n_{C^\circ}}
{dz\wedge d\Omega^{n-2}_{C^\circ}},\qquad
Q_B=\frac{\Omega^n_{B^\circ}}
{dh\wedge d\Omega^{n-2}_{B^\circ}}.
\]
These are \(K_z\)-vector spaces. Localization gives
\[
Q_B\cong\mathcal B_h\otimes_{\CC\{z\}}K_z.
\]
Moreover \(C^\circ=A^\circ\otimes_{K_t}K_z\) and
\(dg=rz^{r-1}dz\). \'Etale base change for differential forms
therefore gives
\[
Q_C\cong
\bigl(\mathcal B_g\otimes_{\CC\{t\}}K_t\bigr)
\otimes_{K_t}K_z.
\]
To check the denominators, observe that
\[
dg\wedge d(a(z)\eta)=a(z)\,dg\wedge d\eta,
\qquad a(z)\in K_z,
\]
because \(dg\wedge dz=0\). The same calculation justifies
localization of each Brieskorn quotient. Consequently
\begin{equation}\label{eq:brieskorn-dimensions}
\dim_{K_z}Q_C=\mu(g),\qquad\dim_{K_z}Q_B=\mu(h).
\end{equation}

For a finite map between smooth spaces, the trace on
differential forms commutes with exterior differentiation,
satisfies the projection formula, and obeys
\[
\Tr_\Phi\circ\Phi^*=\delta\,\mathrm{id}.
\]
These are the trace identities for the de Rham complex; see
\cite[Section~50.19]{Stacks}. In the analytic setting the trace
is the sum over inverse branches off the branch locus. It
extends holomorphically across that locus: at a generic branch
point the local model is a power map in one transverse
coordinate, where summing the branches cancels the apparent
poles; extension across the remaining codimension-two set
follows from normality. Applied to localized forms, this
construction permits poles only along the deleted divisor.
No trace on ordinary forms at the possibly singular origin of
\(C\) is asserted.

Since \(\Phi^*z=h\), pullback and trace send the respective
denominators to one another:
\begin{align*}
\Phi^*(dz\wedge d\eta)&=dh\wedge d(\Phi^*\eta),\\
\Tr_\Phi(dh\wedge d\xi)&=dz\wedge d(\Tr_\Phi\xi).
\end{align*}
They descend to \(K_z\)-linear maps
\[
Q_C\xrightarrow{\Phi^*}Q_B\xrightarrow{\Tr_\Phi}Q_C
\]
whose composition is \(\delta\,\mathrm{id}\). Since
\(\delta\ne0\) in characteristic zero, \(\Phi^*\) is
injective. Equation~\eqref{eq:brieskorn-dimensions} proves
\(\mu(g)\leq\mu(h)\).

Finally, Remark~\ref{rem:isolated-identification} identifies
these Milnor numbers with Hilbert--Samuel multiplicities and
the only possibly nonzero Segre and generic L\^e numbers.
This proves all the numerical conclusions of the theorem.
\end{proof}

\begin{corollary}[A stronger bound for reduced pullback]
\label{cor:isolated-degree}
Under the hypotheses of Theorem~\ref{thm:isolated}, if
\(g\circ F=uh\), then
\begin{equation}\label{eq:isolated-degree-bound}
\mu(h)\geq d\,\mu(g).
\end{equation}
\end{corollary}

\begin{proof}
After normalizing the unit, the chain rule gives
\(J(h)\subseteq J(g)B\). Both ideals are \(\nn\)-primary.
Monotonicity of Hilbert--Samuel multiplicity and
Lemma~\ref{lem:free} yield
\[
\mu(h)=\eHS(J(h);B)
\geq\eHS(J(g)B;B)=d\,\eHS(J(g);A)=d\,\mu(g).
\]
\end{proof}

\begin{remark}[Scope of the algebraic argument]\label{rem:isolated-scope}
For \(r>1\), the ideal \(h^{r-1}J(h)\) is not
\(\nn\)-primary. Thus the chain rule does not permit the
Hilbert--Samuel proof just given. The algebraic proof of
Theorem~\ref{thm:isolated} instead uses Brieskorn modules and
trace; it does not establish a factor \(d\), or \(d/r\),
in \eqref{eq:milnor-bound}.

Both proofs of \eqref{eq:milnor-bound} also apply to reduced,
not necessarily irreducible, \(g,h\) with isolated critical
points, provided the single-power identity
\(g\circ F=uh^r\) is assumed. In the algebraic proof,
Eisenstein may be applied at any irreducible factor of the
reduced germ \(g\). In the topological proof the Milnor
fibers are still connected. This observation will be useful
for transversal slices, which need not be irreducible.
\end{remark}

The theorem of Giraldo--Roeder \cite{GR} implies that a smooth
reduced inverse image cannot lie over a singular hypersurface.
Here we give our independent proof of this result:

\begin{theorem}\label{cormain-theorem} Let $({\bf V},0)\subset (\mathbb{C}^n,0)$ be a germ of a complex hypersurface and let $F: (\mathbb{C}^n,0)\to(\mathbb{C}^n,0)$ be a germ of a finite holomorphic mapping. Then if the germ $(F^{-1}({\bf V}),0)$ taken with the reduced structure is smooth, then also the germ 
$({\bf V}, 0)$ is smooth.
\end{theorem}

\begin{proof}
We proceed by induction with respect to $n.$ Note that if  ${\bf V}$ has an isolated singularity, then it  is smooth if and only if $\mu({\bf V})=0$. In particular in this case  our theorem follows directly from Theorem \ref{thm:isolated}. 
In particular we get the thesis if  $n=2.$  

Now assume that $n>2$ and ${\bf S}:=Sing({\bf V},0)$ has dimension greater than $0.$ Let $\mathcal{L}$ be a linear system of hyperplanes in $\mathbb{C}^n.$ On $V$ let us consider the induced linear system $\mathcal{L}_V$. 
Then a generic member $L_V$ of $\mathcal{L}_V$ is singular, because $S\cap L\not=\emptyset$ (see Lemma 2.5 in \cite{j-b}). Moreover $F^*(L_V)$ (and $F^{-1}(L)$) are  smooth because the system $F^*(\mathcal{L}_V)$ has no base points on smooth $W=F^{-1}(V)\subset F^{-1}(L)$ (the system $F^*(\mathcal{L})$ has no base points on $\mathbb{C}^n$). Since dim $L< n$, this is a contradiction.
\end{proof}

\section{Reduced pullback and nonisolated singularities}
\label{sec:reduced}

Unless otherwise specified, throughout this section
\(g,h\) define singular irreducible reduced hypersurfaces and
\begin{equation}\label{eq:reduced-hypothesis}
g\circ F=uh.
\end{equation}
Use Lemma~\ref{lem:units} to replace \(h\) by \(uh\), and set
\[
J=J(g)\subset A,\qquad I=JB\subset B,\qquad K=J(h)\subset B.
\]
The chain rule gives
\begin{equation}\label{eq:chain-rule}
\nabla h=(DF)^{\mathsf T}(\nabla g)\circ F,
\qquad K\subseteq I.
\end{equation}
Although we use a normalized equation in proofs, all conclusions
about Segre numbers and generic L\^e numbers hold for the original
equation.

\begin{theorem}[Lexicographic monotonicity for reduced pullback]
\label{thm:reduced-lex}
Under \eqref{eq:reduced-hypothesis},
\begin{equation}\label{eq:reduced-lex}
\ee(J(h);B)\geq_{\mathrm{lex}}\ee(J(g);A).
\end{equation}
No equality of the dimensions of \(\Sigma(g)\) and
\(\Sigma(h)\) is required. Equivalently,
\[
\bigl(\lgen{n-1}(h),\ldots,\lgen{0}(h)\bigr)
\geq_{\mathrm{lex}}
\bigl(\lgen{n-1}(g),\ldots,\lgen{0}(g)\bigr).
\]
\end{theorem}

\begin{proof}
By \eqref{eq:chain-rule} and Theorem~\ref{thm:ideal-lex},
\(\ee(K;B)\geq_{\mathrm{lex}}\ee(I;B)\).
Theorem~\ref{thm:flat-segre} gives the componentwise inequalities
\(e_j(I;B)\geq e_j(J;A)\), and hence
\(\ee(I;B)\geq_{\mathrm{lex}}\ee(J;A)\).
Transitivity proves \eqref{eq:reduced-lex}; then apply
\eqref{eq:le-segre-c}.
\end{proof}

\begin{proposition}[The highest-dimensional singular contribution]
\label{prop:top-dimensional}
Assume \eqref{eq:reduced-hypothesis} and
\(\dim\Sigma(g)=\dim\Sigma(h)=s\). Then all coefficients
\(c_i(J(g))\) and \(c_i(J(h))\) with \(i>s\) vanish, and
\begin{equation}\label{eq:s-sandwich}
c_s(J(h);B)\geq c_s(I;B)\geq c_s(J(g);A).
\end{equation}
In Segre and L\^e notation this gives
\[
e_{n-s}(J(h))\geq e_{n-s}(J(g)),\qquad
\lgen{s}(h)\geq\lgen{s}(g).
\]
\end{proposition}

\begin{proof}
Finiteness and faithful flatness give \(\dim B/I=s\).
The vanishing follows from Theorem~\ref{thm:AM}.
Apply Lemma~\ref{lem:first-coefficient} to \(K\subseteq I\)
for the first inequality in \eqref{eq:s-sandwich}, and
Theorem~\ref{thm:flat-segre} for the second.
\end{proof}

\begin{theorem}[Equimultiple target in arbitrary singular dimension]
\label{thm:equimultiple}
Assume
\[
g\circ F=uh,\qquad
\dim\Sigma(g)=\dim\Sigma(h)=s,
\qquad \ell(J(g))=\htop J(g)=n-s.
\]
Then
\begin{equation}\label{eq:equimultiple-componentwise}
e_j(J(h);B)\geq e_j(J(g);A),\qquad 1\leq j\leq n,
\end{equation}
and
\begin{equation}\label{eq:equimultiple-le}
\lgen{i}(h)\geq\lgen{i}(g),\qquad 0\leq i\leq n-1.
\end{equation}
No equimultiplicity assumption on \(J(h)\) is necessary.
If \(J(h)\) is also equimultiple, both multiplicity sequences
are supported only at index \(s\).
\end{theorem}

\begin{proof}
Theorem~\ref{thm:AM} gives \(c_i(J(g))=0\) for \(i\ne s\).
At index \(s\), Proposition~\ref{prop:top-dimensional}
gives the required inequality. At every other index,
nonnegativity of \(c_i(J(h))\) suffices. Equations
\eqref{eq:segre-c} and \eqref{eq:le-segre-c} translate the result.
If \(J(h)\) is equimultiple, apply the same vanishing theorem
to that ideal.
\end{proof}

\begin{remark}[Analytic spread under finite-flat extension]
The intermediate ideal \(I=JB\) has the same
analytic spread as \(J\). Indeed, flatness identifies
\[
I^k/\nn I^k\cong J^k/\mm J^k
\]
for every \(k\), compatibly with products. Thus the special
fiber rings are isomorphic over \(\CC\). In particular,
equimultiplicity of \(J\) implies equimultiplicity of \(I\),
but not of the smaller ideal \(K\).
\end{remark}

\subsection{One-dimensional singular loci}

Suppose now that
\begin{equation}\label{eq:curve-hypothesis}
g\circ F=uh,\qquad\dim\Sigma(g)=\dim\Sigma(h)=1.
\end{equation}
Only \(c_1\) and \(c_0\) can be nonzero. Equivalently, the
only possibly nonzero Segre numbers are \(e_{n-1},e_n\), and
the generic L\^e vector is \((\lgen{1},\lgen{0})\).

\begin{theorem}[An equality criterion in singular dimension one]
\label{thm:curve-equality}
Under \eqref{eq:curve-hypothesis},
\begin{align}
c_1(K)&\geq c_1(I)\geq c_1(J),\label{eq:curve-sandwich}\\
(c_1(K),c_0(K))&\geq_{\mathrm{lex}}(c_1(J),c_0(J)).
\label{eq:curve-lex}
\end{align}
If \(c_1(K)=c_1(I)\), then
\begin{equation}\label{eq:curve-strong-top}
c_0(K)\geq d\,c_0(J).
\end{equation}
In that case all Segre numbers and generic L\^e numbers are
componentwise nondecreasing.
\end{theorem}

\begin{proof}
Equations \eqref{eq:curve-sandwich} and \eqref{eq:curve-lex}
follow from Proposition~\ref{prop:top-dimensional} and
Theorem~\ref{thm:reduced-lex}. Apply Theorem~\ref{thm:ideal-lex}
to \(K\subseteq I\). All earlier Segre coordinates vanish.
If \(c_1(K)=c_1(I)\), lexicographic monotonicity therefore
forces
\[
c_0(K)\geq c_0(I)=d\,c_0(J),
\]
where the last equality is \eqref{eq:j-degree}.
\end{proof}

\begin{corollary}[Componentwise cases in singular dimension one]
\label{cor:curve-cases}
Under \eqref{eq:curve-hypothesis}, assume at least one of:
\begin{enumerate}
\item \(c_1(J(h))=c_1(J(g))\);
\item \(J(g)\) is equimultiple;
\item no irreducible component of \(\Sigma(h)\) is contained
in \(R_F=V(\det DF)\).
\end{enumerate}
Then all Segre numbers and generic L\^e numbers are
componentwise nondecreasing. Moreover,
\[
\lgen{0}(h)=e_n(J(h))\geq d\,e_n(J(g))=d\,\lgen{0}(g).
\]
In condition (3), the assumption that the pullback divisor is
reduced may be replaced by the weaker assumption that its
reduced support is irreducible.
\end{corollary}

\begin{proof}
In case (1), \eqref{eq:curve-sandwich} forces
\(c_1(K)=c_1(I)\), so Theorem~\ref{thm:curve-equality} applies.
In case (2), \(c_0(J)=0\) by Theorem~\ref{thm:AM}; apply
Theorem~\ref{thm:equimultiple} and nonnegativity.

In case (3), at a prime \(\mathfrak q\) defining a curve
component of \(\Sigma(h)\), the matrix \(DF\) is invertible
over \(B_{\mathfrak q}\). The chain rule gives
\(K B_{\mathfrak q}=I B_{\mathfrak q}\). Consequently
\(K\) and \(I\) have the same curve components and the same
local multiplicities on them. Formula~\eqref{eq:AM-assoc}
gives \(c_1(K)=c_1(I)\), and
Theorem~\ref{thm:curve-equality} applies.

For the last assertion, write \(g\circ F=uh^r\). At a point
of \(V(h)\) outside \(R_F\), the map is a local
biholomorphism and the pullback of a reduced divisor is reduced.
Condition (3) supplies such points arbitrarily close to the
origin. Since \(V(h)\) is irreducible, the exponent \(r\)
is constant along it, and hence \(r=1\).
\end{proof}

\begin{remark}[What lexicographic comparison leaves open]
\label{rem:curve-gap}
If \(c_1(K)>c_1(I)\), lexicographic monotonicity for
\(K\subseteq I\) imposes no inequality on \(c_0(K)\).
In particular, \eqref{eq:curve-lex} alone does not prove
\(\lgen{0}(h)\geq\lgen{0}(g)\). This is the remaining
coordinate in singular dimension one outside the cases above.
\end{remark}

\begin{remark}[The total associated-graded multiplicity]
For an ideal \(Q\) with \(\dim R/Q=1\),
\cite[Theorem~3(4)]{AM} gives
\[
\eHS(G_Q(R))=c_0(Q;R)+c_1(Q;R),
\]
where the multiplicity is taken at the homogeneous maximal
ideal of the associated graded ring. Thus, for two such
ideals \(K,L\), the assertion \(c_0(K)\geq c_0(L)\) is
equivalent to
\[
\eHS(G_K(R))-\eHS(G_L(R))\geq c_1(K)-c_1(L).
\]
This is a useful reformulation, not an additional monotonicity
theorem: the required lower bound on the left still needs proof.
\end{remark}

\begin{proposition}[A conditional reduction using isolated perturbations]
\label{prop:ILM}
Let \(g,h\) have one-dimensional critical loci, and choose
generic coordinates \(\mathbf z\) for \(g\) and
\(\mathbf w\) for \(h\). Suppose that, for all sufficiently
large integers \(N\), isolated perturbations satisfying the
L\^e--Iomdine formula also satisfy
\[
\mu(h+b_N w_1^N)\geq\mu(g+a_N z_1^N),
\qquad a_N,b_N\ne0.
\]
Then
\[
(\lgen{1}(h),\lgen{0}(h))
\geq_{\mathrm{lex}}(\lgen{1}(g),\lgen{0}(g)).
\]
\end{proposition}

\begin{proof}
The L\^e--Iomdine formula \cite[Theorem~3.5]{Massey24},
for sufficiently large exponents and suitable nonzero
coefficients, gives
\[
\mu(f+a_N z_1^N)=\lgen{0}(f)+(N-1)\lgen{1}(f).
\]
Subtract the formulas for \(h\) and \(g\). Nonnegativity
for arbitrarily large \(N\) first implies
\(\lgen{1}(h)\geq\lgen{1}(g)\). If these entries agree,
it also implies \(\lgen{0}(h)\geq\lgen{0}(g)\).
\end{proof}

\begin{remark}
The perturbation inequality in Proposition~\ref{prop:ILM}
is an additional hypothesis. The two perturbations need not
be related by a finite pullback, even when \(g\) and \(h\)
are. Also, the pullback of a generic linear coordinate can
have zero differential at the origin. Therefore this
proposition does not prove the missing componentwise case.
\end{remark}

\subsection{A comparison allowing nongeneric prepolar coordinates}

When \(\dim\Sigma(f)=1\), the one-dimensional L\^e cycle is
intrinsic:
\begin{equation}\label{eq:intrinsic-le-cycle}
\Lambda_f^1=\sum_{C\subset\Sigma(f)}\mu_C^\perp(f)[C].
\end{equation}
The sum runs over the irreducible curve components through the
origin, and \(\mu_C^\perp(f)\) is the Milnor number of a
generic transverse slice at a general point of \(C\)
\cite{Massey95,MasseyNotes}. Thus
\[
\lgen{1}(f)=\mult_0\Lambda_f^1.
\]
For a prepolar coordinate system \(\mathbf z\), define its
intersection excess by
\begin{equation}\label{eq:coordinate-excess}
\delta_f(\mathbf z)=
\sum_C\mu_C^\perp(f)
\bigl((C\cdot V(z_1))_0-\mult_0 C\bigr)\geq0.
\end{equation}
The local intersection inequality gives nonnegativity. The
definition of \(\lambda^1\), followed by the Euler characteristic
formula for L\^e numbers \cite[Theorem~4.1]{MasseyNotes}, gives
\begin{equation}\label{eq:prepolar-excess}
\lambda^1_{f,\mathbf z}(0)=\lgen{1}(f)+\delta_f(\mathbf z),
\qquad
\lambda^0_{f,\mathbf z}(0)=\lgen{0}(f)+\delta_f(\mathbf z).
\end{equation}
Indeed, for a one-dimensional critical locus the difference
\(\lambda^0_{f,\mathbf z}(0)-\lambda^1_{f,\mathbf z}(0)\)
is independent of a prepolar choice. This explains precisely
why admissible numerical values need not be generic ones.

\begin{corollary}[A prepolar comparison with zero target excess]
\label{cor:prepolar}
Assume \eqref{eq:curve-hypothesis}, and assume that \(J(g)\)
is equimultiple. Let \(\mathbf z\) be prepolar for \(g\),
with \(\delta_g(\mathbf z)=0\), and let \(\mathbf w\) be
any prepolar system for \(h\). Then
\[
\lambda^i_{h,\mathbf w}(0)\geq\lambda^i_{g,\mathbf z}(0),
\qquad i=0,1.
\]
\end{corollary}

\begin{proof}
Equimultiplicity gives \(\lgen{0}(g)=0\); by
\eqref{eq:prepolar-excess} and \(\delta_g(\mathbf z)=0\),
\(\lambda^0_{g,\mathbf z}(0)=0\) and
\(\lambda^1_{g,\mathbf z}(0)=\lgen{1}(g)\).
The index-zero inequality follows from nonnegativity. For
index one, use \eqref{eq:prepolar-excess} and
Proposition~\ref{prop:top-dimensional}:
\[
\lambda^1_{h,\mathbf w}(0)\geq\lgen{1}(h)
\geq\lgen{1}(g)=\lambda^1_{g,\mathbf z}(0).
\]
\end{proof}

\section{A conditional cycle comparison with higher ramification}
\label{sec:transversal}

Return to \(g\circ F=uh^r\), with \(r\geq1\), and suppose
both critical loci have dimension one. Write
\[
\Sigma(g)=\bigcup_\alpha C_\alpha,\qquad
\Sigma(h)=\bigcup_\beta D_\beta
\]
for their curve components. Unlike the reduced-pullback case,
the inclusion \(J(h)\subseteq J(g)B\) is not available.
The following result isolates sufficient geometric hypotheses.

\begin{proposition}[Transversal comparison and pushforward]
\label{prop:transversal}
Assume the following conditions.
\begin{enumerate}
\item Every \(C_\alpha\) is the image of at least one
\(D_\beta\).
\item Whenever \(F(D_\beta)=C_\alpha\), at general points
\(p\in D_\beta\), \(q=F(p)\in C_\alpha\), there exist
smooth transverse slices \(S_p,T_q\), of dimension \(n-1\),
such that \(F(S_p)\subseteq T_q\) and
\(F|_{S_p}:(S_p,p)\to(T_q,q)\) is finite. The restricted
functions \(h|_{S_p}\), \(g|_{T_q}\) are reduced with
isolated critical points and compute their respective generic
transversal Milnor numbers.
\end{enumerate}
Then
\begin{equation}\label{eq:transverse-mu}
\mu_{D_\beta}^\perp(h)\geq\mu_{C_\alpha}^\perp(g)
\quad\text{whenever }F(D_\beta)=C_\alpha,
\end{equation}
and there is an inequality of effective cycles
\begin{equation}\label{eq:cycle-pushforward}
F_*\Lambda_h^1\geq\Lambda_g^1.
\end{equation}
For a general target linear form \(L\),
\begin{equation}\label{eq:pullback-cut}
\bigl(\Lambda_h^1\cdot V(L\circ F)\bigr)_0
\geq\bigl(\Lambda_g^1\cdot V(L)\bigr)_0
=\lgen{1}(g).
\end{equation}
\end{proposition}

\begin{proof}
Restriction to the compatible slices gives
\[
(g|_{T_q})\circ(F|_{S_p})=(u|_{S_p})(h|_{S_p})^r.
\]
Apply the isolated comparison in the reduced, possibly
reducible, form explained in Remark~\ref{rem:isolated-scope}.
This proves \eqref{eq:transverse-mu}.

Let \(d_\beta\geq1\) be the generic degree of the finite
curve map \(D_\beta\to C_\alpha\). Formula
\eqref{eq:intrinsic-le-cycle} yields
\[
F_*\Lambda_h^1=
\sum_\alpha\left(
\sum_{F(D_\beta)=C_\alpha}
d_\beta\mu_{D_\beta}^\perp(h)\right)[C_\alpha]+E,
\]
where \(E\) is the effective cycle contributed by components
whose images are not components of \(\Sigma(g)\). By
condition (1), each inner sum is nonempty; by
\eqref{eq:transverse-mu}, its coefficient is at least
\(\mu_{C_\alpha}^\perp(g)\). This proves
\eqref{eq:cycle-pushforward}.

Choose \(L\) general enough to meet every image curve
properly and to compute \(\mult_0\Lambda_g^1\). The
projection formula for finite maps gives
\[
\bigl(\Lambda_h^1\cdot V(L\circ F)\bigr)_0
=\bigl(F_*\Lambda_h^1\cdot V(L)\bigr)_0,
\]
because the germ of \(F^{-1}(0)\) is supported at the origin.
Now use \eqref{eq:cycle-pushforward}.
\end{proof}

\begin{corollary}[An additional condition at the origin]
\label{cor:transversal-origin}
Under the hypotheses of Proposition~\ref{prop:transversal},
suppose that one can choose \(L\) as there with
\begin{equation}\label{eq:cut-realizes-multiplicity}
\bigl(\Lambda_h^1\cdot V(L\circ F)\bigr)_0
=\mult_0\Lambda_h^1.
\end{equation}
Then \(\lgen{1}(h)\geq\lgen{1}(g)\). If in addition
\(J(g)\) is equimultiple, all Segre numbers and generic
L\^e numbers are componentwise nondecreasing, for any \(r\).
\end{corollary}

\begin{proof}
The first assertion follows from \eqref{eq:pullback-cut}
and \eqref{eq:cut-realizes-multiplicity}. Equimultiplicity
of \(J(g)\) makes \(\lgen{0}(g)=0\); all other indices
except one vanish. Nonnegativity proves the second assertion.
\end{proof}

\begin{remark}[Why the extra condition is essential to the argument]
A sufficient condition for \eqref{eq:cut-realizes-multiplicity}
is that the linear part of \(L\circ F\) be general for the
tangent cone of \(\Lambda_h^1\). Finiteness alone does not
ensure this. In general \(L\circ F\) is neither a linear
form nor a local coordinate, and its intersection with the
source cycle may exceed the cycle multiplicity. Therefore
\eqref{eq:cycle-pushforward} and \eqref{eq:pullback-cut}
alone do not compare the generic L\^e numbers at the origin.
The compatible smooth slices in condition (2) are also genuine
extra hypotheses, not an automatic property of a finite map.
\end{remark}

\section{Examples and limits of the comparison statements}
\label{sec:examples}

\begin{example}[Equimultiple examples in every singular dimension]
\label{ex:equimultiple}
Fix \(n\geq2\), \(0\leq s\leq n-2\), and put \(m=n-s\).
With \(s\) additional variables \(\mathbf z\), let
\begin{align*}
g(\mathbf x,\mathbf z)&=x_1^2+x_2^2+\cdots+x_{m-1}^2+x_m^3,\\
F(\mathbf u,\mathbf w)&=(u_1^2,u_2,\ldots,u_m,\mathbf w),\\
h(\mathbf u,\mathbf w)&=u_1^4+u_2^2+\cdots+u_{m-1}^2+u_m^3.
\end{align*}
Then \(d=2\), \(g\circ F=h\), and both hypersurfaces are
irreducible. For \(m=2\) this follows from the coprime
exponent pairs \((2,3)\), \((4,3)\); for \(m>2\) it follows
because the transverse hypersurfaces have isolated singularities
and dimension at least two. Both singular loci are smooth of
dimension \(s\).

The Jacobian ideals are complete intersections of height \(m\):
\[
J(g)=(x_1,\ldots,x_{m-1},x_m^2),\qquad
J(h)=(u_1^3,u_2,\ldots,u_{m-1},u_m^2).
\]
They are equimultiple. The transverse quotient lengths are
\(2\) and \(6\), while the singular-locus multiplicities are
one. Theorem~\ref{thm:AM} gives
\[
c_s(J(g))=2,\qquad c_s(J(h))=6,
\qquad c_i(J(g))=c_i(J(h))=0\quad(i\ne s).
\]
Thus \(\lgen{s}\) increases from \(2\) to \(6\), and all
other generic L\^e numbers vanish. If \(s>0\), the entire
singular locus upstairs lies in \(R_F=V(u_1)\). This example
is covered by the equimultiple theorem, not by the
nonramification condition of Corollary~\ref{cor:curve-cases}.
\end{example}

\begin{example}[Degree can occur entirely in the singular direction]
\label{ex:tangential-degree}
In \(\CC^3\), take
\[
g(x,y,z)=x^2+y^3,\qquad F(u,v,w)=(u,v,w^d),\qquad d>1.
\]
The reduced pullback is \(h=u^2+v^3\), so
\[
\ee(J(g))=\ee(J(h))=(0,2,0),\qquad
(\lgen{1},\lgen{0})=(2,0).
\]
Both hypersurfaces are irreducible and both singular loci are
one-dimensional. Thus neither a factor \(d\) nor a factor
\(d/r\) can be inserted in every coordinate of the proposed
nonisolated comparison: here \(r=1\) and that factor would
require \(2\geq2d\). Adding unused variables gives the same
phenomenon in higher singular-locus dimension.
\end{example}

\begin{example}[A non-equimultiple pair with one-dimensional singular locus]
\label{ex:quartic}
Consider
\[
g(x,y,z)=x^4+y^4-xyz,\qquad
F(x,y,z)=(x,y,z^2),\qquad
h(x,y,z)=x^4+y^4-xyz^2.
\]
Then \(d=2\), \(g\circ F=h\), and
\(\Sigma(g)=\Sigma(h)=V(x,y)\). At a general point of this
axis, each transverse singularity is a node, so its Milnor
number is one. Therefore
\begin{equation}\label{eq:quartic-first}
\lgen{1}(g)=\lgen{1}(h)=1.
\end{equation}

Both equations are irreducible. For \(g\), a factorization
would have a factor of order one, with initial form \(x\),
\(y\), or \(z\), because its initial form is \(-xyz\).
An implicit solution \(x=\varphi(y,z)\) with
\(\operatorname{ord}\varphi\geq2\) contradicts
\(\varphi(y,0)^4+y^4=0\); the \(y\)-case is identical.
An implicit solution for \(z\) would require
\(x^4+y^4\) to be divisible by \(xy\), which it is not.
For \(h\), regard it as a primitive quadratic in \(z\):
\((x^4+y^4)/(xy)\) is not a square in \(\CC(x,y)\), since
its valuation at \(x\) is odd. Hence \(h\) is irreducible
as a homogeneous polynomial and therefore as an analytic germ.

The remaining numbers can be computed explicitly from L\^e
cycles with first coordinate \(z\). For \(g\), the residual
polar curve is the union of four reduced curves
\[
y=\zeta x,\qquad z=4\zeta^{-1}x^2,\qquad \zeta^4=1.
\]
Along each curve, \(g_z=-xy\) has order two, giving
\(\lambda^0_{g,(z,x,y)}(0)=8\). For \(h\), the residual
polar curve consists of eight reduced lines
\[
y=\zeta x,\qquad z^2=4\zeta^{-1}x^2,\qquad \zeta^4=1.
\]
On each line \(h_z=-2xyz\) has order three, giving
\(\lambda^0_{h,(z,x,y)}(0)=24\).

These coordinate choices are prepolar: their first slices are
\(x^4+y^4\), with isolated critical point. The coordinate
\(z\) meets the singular axis with multiplicity one, so the
excess in \eqref{eq:coordinate-excess} is zero. Thus the
computed numbers are the generic values, and
\begin{equation}\label{eq:quartic-numbers}
\ee(J(g))=(0,1,8),\qquad
\ee(J(h))=(0,1,24).
\end{equation}
In particular, both ideals have maximal analytic spread, and
Corollary~\ref{cor:curve-cases}(1) applies. This example
illustrates a componentwise case beyond equimultiplicity.
\end{example}

\begin{example}[Prepolar values need not be generic values]
\label{ex:prepolar-counterexample}
For \(k\geq2\), let
\[
f(x,y,z)=(x-z^k)^2+y^3.
\]
It is an irreducible reduced germ, analytically equivalent to
a cusp times a line. Its Jacobian ideal and singular locus are
\[
J(f)=(x-z^k,y^2),\qquad
C=\Sigma(f)=V(x-z^k,y).
\]
The ideal is equimultiple, \(C\) is smooth, and the
transverse Milnor number is two. Hence
\[
(\lgen{1}(f),\lgen{0}(f))=(2,0).
\]

The linear coordinate systems \(\mathbf z=(x,y,z)\) and
\(\mathbf w=(z,y,x)\) are prepolar. Their first slices are
respectively \(z^{2k}+y^3\) and \(x^2+y^3\), both with
isolated critical point. For \(\mathbf z\),
\[
\Lambda^1_{f,\mathbf z}=2[C],\qquad
\Gamma^1_{f,\mathbf z}=2(k-1)[V(y,z)].
\]
Intersecting the first cycle with \(V(x)\), and the second
with \(V(f_x)\), gives
\[
(\lambda^1_{f,\mathbf z}(0),\lambda^0_{f,\mathbf z}(0))
=(2k,2k-2).
\]
For \(\mathbf w\), the residual polar curve vanishes and
the corresponding pair is \((2,0)\).

Taking \(F=\mathrm{id}\), \(g=h=f\), using \(\mathbf z\)
on the target and \(\mathbf w\) on the source, disproves a
comparison asserted for \emph{arbitrary} prepolar choices.
It does not contradict the generic formulation: the two
generic vectors are equal.
\end{example}

\begin{remark}[Multiplying the Jacobian ideal by a power of its equation]
\label{rem:power-ideal}
There is no general identity
\(c_i(h^{r-1}J(h))=r\,c_i(J(h))\). For example, in
\(\CC^2\) let \(h=x^2+y^3\) and \(r\geq2\). Since
\(J(h)=(x,y^2)\) is primary, \(e_1(J(h))=0\). But the
divisorial Segre cycle of \(h^{r-1}J(h)\) is
\((r-1)[V(h)]\), so
\[
e_1(h^{r-1}J(h))=2(r-1)>0.
\]
This is why the factor \(h^{r-1}\) in the chain rule
cannot be removed by a uniform rescaling of the multiplicity
sequence.
\end{remark}

\section{Conjectures and remaining questions}\label{sec:conjectures}

\begin{conjecture}[Componentwise Segre monotonicity]
\label{conj:segre-monotonicity}
Let \(F:(\CC^n,0)\to(\CC^n,0)\) be finite, and let
\(V=V(g)\) be a singular irreducible reduced hypersurface.
Suppose \(W=F^{-1}(V)_{\mathrm{red}}\) is irreducible and
let \(h\) be a reduced equation for \(W\). Then
\[
e_j(J(h);\CC^n)\geq e_j(J(g);\CC^n),\qquad 1\leq j\leq n.
\]
\end{conjecture}

\begin{conjecture}[Equivalent generic L\^e formulation]
\label{conj:le-monotonicity}
Under the hypotheses of Conjecture~\ref{conj:segre-monotonicity},
\[
\lgen{i}(h)\geq\lgen{i}(g),\qquad 0\leq i\leq n-1.
\]
\end{conjecture}

The two conjectures are equivalent by
\eqref{eq:le-segre-c}; they are not independent assertions.
They hold in the isolated case by Theorem~\ref{thm:isolated}.
For reduced pullback, they hold under the common-dimension
equimultiplicity hypothesis of Theorem~\ref{thm:equimultiple}
and under the one-dimensional criteria of
Corollary~\ref{cor:curve-cases}. Under higher ramification,
Corollary~\ref{cor:transversal-origin} supplies an additional
conditional equimultiple case.

\begin{conjecture}[Lexicographic comparison with arbitrary ramification]
\label{conj:ramified-lex}
Under the hypotheses of Conjecture~\ref{conj:segre-monotonicity},
\[
\ee(J(h))\geq_{\mathrm{lex}}\ee(J(g)).
\]
\end{conjecture}

This is weaker than componentwise comparison. It is proved
here for \(r=1\) by Theorem~\ref{thm:reduced-lex}, and for
isolated singularities with arbitrary \(r\) by
Theorem~\ref{thm:isolated}. No proof for general nonisolated
singularities with \(r>1\) is claimed.

The remaining issues can be stated concretely. With reduced
pullback and one-dimensional singular loci, a strict increase
of \(c_1\) leaves the comparison of \(c_0\) unresolved outside
the sufficient conditions above. In higher singular dimension,
lexicographic monotonicity similarly leaves later coordinates
uncontrolled after the first strict increase. With higher
ramification, one must additionally overcome the factor
\(h^{r-1}\) in the chain rule or obtain compatible geometric
comparisons beyond those in Section~\ref{sec:transversal}.
Neither equality of radicals nor the projection formula for
a nongeneric cut resolves these issues.

\end{document}